\documentclass[a4paper,12pt]{elsarticle}
\usepackage{amsmath,amssymb,hyperref,graphicx,amsthm}

\makeatletter
\def\ps@pprintTitle{%
  \let\@oddhead\@empty
  \let\@evenhead\@empty
  \def\@oddfoot{\reset@font\hfil\thepage\hfil}
  \let\@evenfoot\@oddfoot
}
\makeatother

\usepackage{amsmath,amssymb,graphicx,amsthm}

\newtheorem{theorem}{Theorem}[section]
\newtheorem*{theorem:repeat}{\tref{butterflystab}}
\newtheorem*{theorem:repeatmain}{\tref{main}}
\newtheorem{lemma}[theorem]{Lemma}
\newtheorem{corollary}[theorem]{Corollary}

\newtheorem{prop}[theorem]{Proposition}
\newtheorem{conjecture}[theorem]{Conjecture}
\newtheorem{claim}[theorem]{Claim}

\newcommand\lref[1]{Lemma~\ref{lem:#1}}
\newcommand\tref[1]{Theorem~\ref{thm:#1}}
\newcommand\cref[1]{Corollary~\ref{cor:#1}}

\newcommand\pref[1]{Proposition~\ref{prop:#1}}

\newcommand\cC{{\mathcal C}}

\newcommand\cE{{\mathcal E}}
\newcommand\cF{{\mathcal F}}
\newcommand\cG{{\mathcal G}}
\newcommand\cH{{\mathcal H}}

\newcommand\cM{{\mathcal M}}

\newcommand\cS{{\mathcal S}}

\numberwithin{equation}{section}

\begin{document}
\begin{frontmatter}
\title{Supersaturation and stability for forbidden subposet problems}

\author{Bal\'azs Patk\'os}
\address{MTA--ELTE Geometric and Algebraic Combinatorics Research Group,\\ H--1117 Budapest, P\'azm\'any P.\ s\'et\'any 1/C, Hungary.  \\ Alfr\'ed R{\'e}nyi Institute of Mathematics, Hungarian Academy of Sciences \\ H-1053, Budapest, Re\'altanoda u. 13-15, Hungary\\
        {\tt email}: patkosb@cs.elte.hu, patkos.balazs@renyi.mta.hu}

\begin{abstract}
We address a supersaturation problem in the context of forbidden subposets. A family $\cF$ of sets is said to contain the poset $P$ if there is an injection $i:P \rightarrow \cF$ such that $p \le_P q$ implies $i(p) \subset i (q)$. The poset on four elements $a,b,c,d$ with $a,b \le c,d$ is called a butterfly. The maximum size of a family $\cF \subseteq 2^{[n]}$ that does not contain a butterfly is $\binom{n}{\lfloor n/2 \rfloor}+\binom{n}{\lfloor n/2 \rfloor+1}$ as proved by De Bonis, Katona, and Swanepoel. We prove that if $\cF \subseteq 2^{[n]}$ contains $\binom{n}{\lfloor n/2 \rfloor}+\binom{n}{\lfloor n/2 \rfloor+1}+E$ sets, then it has to contain at least $(1-o(1))E(\lceil n/2 \rceil +1)\binom{\lceil n/2\rceil}{2}$ copies of the butterfly provided $E\le 2^{o(n)}$. We show that this is asymptotically tight and for small values of $E$ we show that the minimum number of butterflies contained in $\cF$ is exactly $E(\lceil n/2 \rceil +1)\binom{\lceil n/2\rceil}{2}$.
\end{abstract}
\end{frontmatter}

\section{Introduction}

Many problems in extremal combinatorics deal with determining the maximum size $M$ that certain combinatorial structures can have provided they satisfy some prescribed property. Many properties are defined via forbidden substructures. Therefore, for any $M'>M$ it is natural to ask what is the minimum number of forbidden substructures that appear in a structure of size $M'$. Such problems are usually called \textit{supersaturation} or \textit{Erd\H os-Rademacher type} problems. The first such result is a strengthening of Mantel's theorem \cite{M} which states that a triangle-free graph on $n$ vertices may contain at most $\lfloor \frac{n^2}{4}\rfloor$ edges. Rademacher proved that any graph on $n$ vertices with $\lfloor \frac{n^2}{4}\rfloor+1$ edges contains at least $\lfloor n/2 \rfloor$ triangles and later Erd\H os \cite{E2} improved this to $t\cdot\lfloor n/2 \rfloor$ triangles for graphs with $\lfloor \frac{n^2}{4}\rfloor+t$ edges provided $t \le cn$, where $c$ is a sufficiently small positive constant. The general $r$-clique case was settled asymptotically (after the work of many researchers) by Reiher \cite{R}. A supersaturation phenomenon that holds for both graphs and hypergraphs was discovered by Erd\H os and Simonovits \cite{ES}: for any $k$-uniform hypergraph $F$ let $ex_k(n,F)$ denote the maximum number of hyperedges in a hypergraph $G$ on $n$ vertices that does not contain $F$. If a $k$-uniform hypergraph $H$ contains $ex_k(n,F)+cn^k$ hyperedges, then $H$ contains at least $c'n^{|V(F)|}$ copies of $F$.

One of the first appearances of supersaturation problems in extremal set theory was the minimization problem of the number of disjoint pairs of sets in a family $\cF \subseteq 2^{[n]}$ of size $m$. As observed by Erd\H os, Ko and Rado \cite{EKR} a maximum \textit{intersecting} family (a family of sets such that all pairwise intersections are non-empty) has size $2^{n-1}$, thus the problem of determining the minimum number of disjoint pairs is interesting if $m>2^{n-1}$. Ahlswede \cite{A} and Frankl \cite{F} independently proved that an optimal family must contain $\binom{[n]}{\ge l+1}$ and must be contained in $\binom{[n]}{\ge l}$ for the integer $l$ defined by $\sum_{i=l+1}^n\binom{n}{i}\le m \le \sum_{i=l}^n\binom{n}{i}$. As any subset $F$ of $[n]$ of size $l$ is disjoint from the same number of subsets of size at least $l+1$, this reduces the problem to finding an $l$-uniform family $\cF\subseteq \binom{[n]}{l}$ of size $m$ that minimizes the number of disjoint pairs in $\cF$. By the celebrated theorem of Erd\H os, Ko and Rado \cite{EKR}, a maximum intersecting family $\cF\subset \binom{[n]}{l}$ has size $\binom{n-1}{l-1}$ provided $2l \le n$ holds. Thus, the supersaturation problem is to minimize the number of disjoint pairs for uniform families of size $m$ where $m>\binom{n-1}{l-1}$. This was first addressed by Bollob\'as and Leader \cite{BL} and a major improvement was achieved recently by Das, Gan, and Sudakov \cite{DGS}. (The graph case $l=2$ was first solved by Ahlswede and Katona \cite{AK} in a different context.)

The first theorem in extremal set theory is that of Sperner \cite{S}. It states that a family $\cF \subseteq 2^{[n]}$ for which there does not exist a pair $F,G\in \cF$ with $F \subsetneq G$ can have size at most $\binom{n}{\lfloor n/2\rfloor}$. These families are called \textit{Sperner families}. A generalization is due to Erd\H os \cite{E1}: a \textit{$k$-Sperner family} is one that does not contain a chain $F_1 \subsetneq F_2 \subsetneq \dots \subsetneq F_{k+1}$ of length $k+1$. He proved that if $\cF\subseteq 2^{[n]}$ is $k$-Sperner, then $|\cF|\le \Sigma(n,k):=\sum_{i=1}^k\binom{n}{\lfloor \frac{n-k}{2}\rfloor+i}$ holds and also that the only $k$-Sperner families with this size are $\cup_{i=1}^k\binom{[n]}{\lfloor \frac{n-k}{2}\rfloor+i}$ and $\cup_{i=0}^{k-1}\binom{[n]}{\lceil \frac{n-k}{2}\rceil+i}$ (these two families coincide if $n+k$ is odd). We denote the set of these families by $\Sigma^*(n,k)$. The corresponding supersaturation problem (i.e. minimizing the number of $(k+1)$-chains for families $\cF\subseteq 2^{[n]}$ with $|\cF|=m>\Sigma(n,k)$) was addressed by Kleitman \cite{Kl} for the case $k=1$. For larger values of $k$, the problem was recently settled by Das, Gan, and Sudakov \cite{DGS} and independently by Dove, Griggs, Kang, and Sereni \cite{DGKS} for some range of family sizes.

Sperner's and Erd\H os's theorems can be formulated in the following more general context. Let $P$ be a finite poset. We say that a family $\cF$ of sets contains $P$ if there exists an injection $i: P \rightarrow \cF$ such that $p\le_P q$ implies $i(p) \subset i(q)$. A family $\cF$ is \textit{$P$-free} if it does not contain $P$. The size of a largest $P$-free family $\cF \subseteq 2^{[n]}$ is denoted by $La(n,P)$. Note that Erd\H os's theorem can be restated as $La(n,P_{k+1})=\Sigma(n,k)$, where $P_{k+1}$ denotes the total ordering or path on $k+1$ elements. There are not too many posets $P$ for which the exact value of $La(n,P)$ is known. One such example is the butterfly poset $B$ on four elements $a,b,c,d$ with $a\le_B c$, $a\le_Bd$, $b\le_B c$, $b\le_B d$.

\begin{theorem}[De Bonis, Katona, Swanepoel \cite{DKS}]
\label{thm:butt}
If $\cF \subseteq 2^{[n]}$ is a butterfly-free family, then $|\cF|\le \Sigma(n,2)$ holds and equality holds if and only if $\cF \in \Sigma^*(n,2)$.
\end{theorem}

In this paper we will address the problem of minimizing the number of butterflies contained in families $\cF \subseteq 2^{[n]}$ of fixed size $m>\Sigma(n,2)$. It will be more convenient to count different image sets of injections of $B$ to $\cF$ as \textit{copies} of $B$ instead of counting the number of injections. Also,the problems are equivalent as there are exactly four injections (the number of automorphisms of $B$) to any possible image set.

Note that whenever we add a set $G$ to a family $\cF \in \Sigma^*(n,2)$, the number of newly constructed butterflies in $\cF \cup \{G\}$ will be minimized if $G$ is ``closest to the middle". If $n=2k$ and $\cF=\binom{[n]}{k-1}\cup\binom{[n]}{k}$, then $G$ should be picked from $\binom{[n]}{k+1}$. In this case, if $F_1,F_2,F_3,G$ is a newly created butterfly, then $F_1,F_2 \subset F_3 \subset G$ must hold. If one adds a set $G$ to a family $\cF$ from $\Sigma^*(n,2)$ with $|G|>k+1$, then the number of butterflies with $F_1,F_2 \subset F_3 \subset G$, $F_i \in \cF$ is already larger than in the previous case. Thus, independently of parity, the minimum number of butterflies appearing when adding one new set to a family in $\Sigma^*(n,2)$ is $$f(n)=(\lceil n/2 \rceil +1)\binom{\lceil n/2\rceil}{2}.$$ Therefore, if adding $E$ new sets to a family $\cF \in \Sigma^*(n,2)$ we will have at least $E\cdot f(n)$ butterflies. Note that if $G_1,G_2  \in \binom{[n]}{k+1}$ are such that $|G_1 \cap G_2| \le k-1$, then there are no butterflies in $\cF \cup \{G_1,G_2\}$ that contain both $G_1$ and $G_2$. Thus it is possible to have only $E\cdot f(n)$ copies of butterfly as long as we can pick sets from $\binom{[n]}{k+1}$with this property. We summarize our findings in the following proposition.

\begin{prop}
\label{prop:constr} (a) If $\cS\subset\cF$ for some $\cS \in \Sigma^*(n,2)$, then $\cF$ contains at least $(|\cF|-\Sigma(n,2))f(n)$ copies of butterflies.

(b) If $\cF=\binom{[n]}{\lceil n/2\rceil-1}\cup\binom{[n]}{\lceil n/2\rceil} \cup \cE$ where $\cE \subset \binom{[n]}{\lceil n/2\rceil+1}$ such that $|E_1 \cap E_2|< \lceil n/2\rceil$ holds for all $E_1,E_2 \in \cE$, then $\cF$ contains exactly $|\cE|\cdot f(n)$ copies of butterflies.
\end{prop}

It is known that it is possible to construct a family $\cE$ with the above property as long as the number of sets in $\cE$ is not more than $\frac{1}{n}\binom{n}{k+1}$: the families $\cE_j=\{E\in\binom{[n]}{k+1}: \sum_{i\in E}i \equiv j (\text{mod} n)\}$ all possess this property, therefore the largest among them must be of size at least $\frac{1}{n}\binom{n}{k+1}$. The main result of this paper states that this is best possible for all families of size $\Sigma(n,2)+E$, if $E$ is very small and asymptotically best possible, if $E$ is not that small.

\begin{theorem}
\label{thm:main} Let $\cF\subseteq 2^{[n]}$ be a family of sets with $|\cF|=\Sigma(n,2)+E$.

(a) If $E=E(n)$ satisfies $\log E=o(n)$, then the number of butterflies contained by $\cF$ is at least $(1-o(1))E\cdot f(n)$. 

(b) Furthermore, if $E \le \frac{n}{100}$, then the number of butterflies contained by $\cF$ is at least $E\cdot f(n)$.
\end{theorem}

One of our major tools in proving \tref{main} will be the following stability version of \tref{butt}

\begin{theorem}
\label{thm:butterflystab} Let $m$ be a non-negative integer with $m\le \binom{\frac{2n}{3}-1}{\lceil n/2\rceil}$ and let $\cF \subseteq 2^{[n]}$ be a butterfly-free family such that $|\cF \setminus \cF^*|\ge m$ for every $\cF^*\in \Sigma^*(n,2)$. Then the inequality $|\cF| \le \Sigma(n,2)-\frac{m}{4}$ holds if $n$ is large enough.
\end{theorem}

Note that the conditions of \tref{main} and \tref{butterflystab} imply that for any fixed positive $\varepsilon$ both $m$ and $E$ are not larger than $\binom{\varepsilon n}{\varepsilon n/2}$ provided $n$ is large enough.

The rest of the paper is organized as follows. In Section 2, we prove \tref{butterflystab}, then in Section 3 we prove \tref{main}. Section 4 contains some concluding remarks.

\section{Stability of maximum butterfly-free families}

In this section we prove \tref{butterflystab}, a stability version of \tref{butt}. As butterfly-free families of maximum size possess the 2-Sperner property, the proof of \tref{butterflystab} consists of two steps: we first establish a stability result (\lref{spernerstab}) on 2-Sperner families and then we consider how the number of 3-chains in a butterfly-free family $\cF$ affect the size of $\cF$. Note that if $n$ is odd, then \lref{spernerstab} would easily follow from the characterization of 1-Sperner families as in that case the unique family in $\Sigma^*(n,2)$ consists of two Sperner families of equal maximum size. However, if $n$ is even we need a little work to obtain the same result. We start with stating the celebrated LYM-inequality \cite{Bol, Lub, Mes, Y}. This was originally stated for Sperner families, but using the fact that any $k$-Sperner family can be decomposed into $k$ antichains, the statement generalizes easily to $k$-Sperner families. As we will not need the result in its full generality, we state it in the case $k=2$. (For more on Sperner properties and the LYM-inequality, see \cite{E}.)

\begin{theorem}[LYM-inequality for 2-Sperner families]
\label{thm:lym}
If $\cF \subseteq 2^{[n]}$ is a 2-Sperner family, then the inequality
\[
\sum_{F\in\cF}\frac{1}{\binom{n}{|F|}}\le 2
\]
holds.
\end{theorem}

\begin{corollary}
\label{cor:weakstab}
Let $\cF \subseteq 2^{[n]}$ be a 2-Sperner family such that one of the following holds:

(a) $n$ is odd and the number of sets $|\{G \notin \cF: |G|=\lceil n/2\rceil \text{or}\ |G|=\lfloor n/2\rfloor\}|$ is at least $m$,
 
(b) $n$ is odd and the number of sets $|\{F \in \cF: |F|\neq \lceil n/2\rceil, \lfloor n/2\rfloor\}|$ is at least $m$,

(c) $n$ is even and the number of sets $|\{G \notin \cF: |G|= n/2\}|$ is at least $m$,

(d) $n$ is even and the number of sets $|\{F \in \cF: |F|\neq  n/2-1, n/2, n/2+1\}|$ is at least $m$.
 
Then we have the inequality $|\cF|\le \Sigma(n,2)-\frac{1.9m}{n}$.
\end{corollary}

\begin{proof}
Parts (a),(b),(c) immediately follow from the LYM-inequality and from the fact that the ratio of the two smallest possible terms in the Lubell-function is 
$$\binom{n}{\lfloor n/2\rfloor-1}^{-1}/\binom{n}{\lfloor n/2\rfloor}^{-1}=\frac{\lfloor n/2\rfloor+1}{\lfloor n/2\rfloor}\ge 1+\frac{1.9}{n}.$$
To obtain (d), observe that the ratio of the second and third smallest possible terms in the Lubell-function is
$$\binom{n}{n/2-2}^{-1}/\binom{n}{n/2-1}^{-1}=\frac{n/2+2}{n/2-1}\ge 1+\frac{1.9}{n}. \hfill  \qedhere $$
\end{proof}

We continue with introducing the notions of \textit{shadow} and \textit{shade}. If $\cF$ is a family of sets, then its $k$-shadow is $\Delta_k(\cF)=\{G: |G|=k, \exists F \in\cF \hskip 0.2truecm G \subset F\}$. To define the $k$-shade of the family we have to assume the existence of an underlying set, say $\cF \subseteq 2^{[n]}$. If so, then the $k$-shade is defined as $\nabla_k(\cF)=\{G \in \binom{[n]}{k}: \exists F \in \cF\hskip 0.2truecm F \subset G\}$. The well-known theorem of Kruskal \cite{Kr} and Katona \cite{Kat} states which family $\cF$ of $k$-sets minimizes the size of $\Delta_{k-1}(\cF)$ among all families of $m$ sets. For calculations the following version happens to be more useful than the precise result. To state the theorem we need to introduce the polynomial $\binom{x}{k}=\frac{x\cdot(x-1)\cdot ... \cdot(x-k+1)}{k!}$.

\begin{theorem}[Lov\'asz, \cite{L}]
\label{thm:shadow}
Let $\cG$ be a family of $k$-sets and let $x\ge k$ be the real number such that $\binom{x}{k}=|\cG|$ holds. Then the family of shadows satisfies $|\Delta_{k-1}(\cG)| \ge \binom{x}{k-1}$
\end{theorem}

We will apply \tref{shadow} in a slightly more general setting. If $\cF \subset \binom{[n]}{l}$ with $l> n/2$, then a simple double counting argument and Hall's theorem show that there exists a matching from $\cF$ to $\Delta_{l-1}(\cF)$ such that if $F \in\cF$ and $G \in \Delta_{l-1}(\cF)$ are matched, then $G \subset F$. Using this observation and \tref{shadow} one obtains the following lemma. Part (ii) of the statement follows from the fact that $G \subset F \subset [n]$ holds if and only if $[n]\setminus F \subset [n] \setminus G$.

\begin{lemma}
\label{lem:genshadow}
(a) Let $\cG \subseteq \binom{[n]}{\ge k}$ be a Sperner family with $\lfloor n/2\rfloor \le k$ and let $x\ge k$ be the real number such that $\binom{x}{k}=|\cG|$ holds. Then the family of shadows satisfies $|\Delta_{k-1}(\cG)| \ge \binom{x}{k-1}$.

(b) Let $\cG \subseteq \binom{[n]}{\le k}$ be a Sperner family with $\lceil n/2\rceil \ge k$ and let $x\ge k$ be the real number such that $\binom{x}{n-k}=|\cG|$ holds. Then the family of shades satisfies $|\nabla_{k+1}(\cG)| \ge \binom{x}{n-k-1}$.
\end{lemma}


Let us define the following functions of $l$ and $m$. Let $x=x(l,m)$ be defined by the equation $\binom{x}{l}=m$ and write $g(l,m)=\binom{x}{l-1}-\binom{x}{l}$. According to \lref{genshadow}, if $\cG \subseteq \binom{[n]}{\ge l}$ is a Sperner family with $|\cG|=m$ and $l\ge \lfloor n/2\rfloor$, then $|\Delta_{l-1}(\cG)|-|\cG|\ge g(l,m)$ holds. This will be crucial in the proof of \lref{spernerstab}, the stability result on 2-Sperner families. When comparing the size of a 2-Sperner family $\cF$ to $\Sigma(n,2)$ we will split $\cF$ into an upper and lower Sperner family $\cF_u$ and $\cF_l$. Then $\cF_u'=\{F\in \cF_u: |F|>n/2\}$ will be replaced by $\Delta_{n/2}(\cF'_u)$ and $\cF_l'=\{F\in \cF_l: |F|>n/2-1\}$ will be replaced by $\Delta_{n/2-1}(\cF'_l)$. As the resulting family is 2-Sperner, we will have $|\cF|\le \Sigma(n,2)-g(n/2,|\cF'_u|)-g(n/2-1,|\cF'_l|)$. To be able to do calculations with expressions involving the $g$ function, we gather some properties of $x(l,m)$ and $g(l,m)$ in the following proposition.

\begin{prop}
\label{prop:change}
(a)  If $m \le \binom{2l}{l}$, then $x(l,m)\le x(l+1,m) \le x(l,m)+1$ holds.

(b) If $x(l,m) \le 2l-1$, then $g(l,m) \ge 0$ holds.

(c) If $m_1+m_2=m$ and $x(l,m)\le 2l-1$, then $g(l,m_1)+g(l,m_2)\ge g(l,m)$ holds.

(d) If $m \le \binom{2l-1}{l}$, then $g(l,m)\le g(l+1,m)$ holds.

(e) For every $\varepsilon>0$ there exists $l_0$ such that if $l\ge l_0$, then $g(l,m)$ is increasing in the interval $0\le m\le \binom{(2-\varepsilon)l}{l}$.

(f) If $x(l,m) \le 4l/3-1$, then $2m \le g(l,m)\le 2lm$ holds.

\end{prop}

\begin{proof} (a) Clearly the polynomial $\binom{x}{l+1}$ is monotone increasing in $x$ if $x\ge l+1$ holds. Observe that $\binom{x(l,m)}{l+1}=\frac{x(l,m)-l}{l+1}\binom{x(l,m)}{l}<m$ as $x(l,m)\le 2l$ by the assumption $m \le \binom{2l}{l}$. Therefore, $x(l,m)\le x(l+1,m)$ holds. Similarly, $\binom{x(l,m)+1}{l+1}=\frac{x(l,m)+1}{l+1}\binom{x(l,m)}{l}>m$ and therefore $x(l+1,m)\le x(l,m)+1$ holds.

To obtain (b), (c), and (f) write $g(l,m)$ in the following form
\[
g(l,m)=\binom{x}{l-1}-\binom{x}{l}=\left(\frac{l}{x-l+1}-1\right)\binom{x}{l}=\frac{2l-x-1}{x-l+1}m.
\]
(b) and (f) are straightforward and to obtain (c) note that as for fixed $l$ we know that $x(l,m)$ is an increasing function of $m$, the fraction $\frac{2l-x-1}{x-l+1}$ is decreasing in $m$.

To obtain (d), as $g(l,m)-g(l+1,m)=\binom{x(l,m)}{l-1}-\binom{x(l+1,m)}{l}$ we need to compare $\binom{x(l,m)}{l-1}$ and $\binom{x(l+1,m)}{l}$.
\[
\frac{\binom{x(l,m)}{l-1}}{\binom{x(l+1,m)}{l}}=\frac{\binom{x(l,m)}{l-1}}{m}\frac{m}{\binom{x(l+1,m)}{l}}=\frac{\binom{x(l,m)}{l-1}}{\binom{x(l,m)}{l}}\frac{\binom{x(l+1,m)}{l+1}}{\binom{x(l+1,m)}{l}}=\frac{l}{x(l,m)-l+1}\cdot\frac{x(l+1,m)-l}{l+1}<1,
\]
where we used $x(l+1,m) \le x(l,m)+1$ of (a).

To obtain (e) consider $g(l,m)$ in the following form
\[
g(l,m)=\binom{x}{l-1}-\binom{x}{l}=\frac{x(x-1)\dots (x-l+2)(2l-x-1)}{l!}.
\]
As a function of $x$ it is a polynomial with no multiple roots, therefore between $l-2$ and $2l-1$ it is a concave function with one maximum. Its derivative is 
\[
\frac{1}{l!}\left((2l-x-1)\sum_{i=0}^{l-2}\prod_{j=0,j\ne i}^{l-2}(x-j) -\prod_{i=0}^{l-2}(x-i)\right).
\]
If $x \le (2-\varepsilon)l$, then any product in the sum is at least an $\varepsilon$ fraction of the product to subtract. Thus if $l$ is large enough the derivative is positive and thus the function is increasing. As $x$ is a monotone increasing function of $m$, the claim holds.
\end{proof}

After all these preliminary results we are ready to state and prove the stability result on 2-Sperner families. Let us remind the reader that we would like to obtain a lemma that states that if a 2-Sperner family $\cF$ is very different from the largest one(s) (i.e. that/those in $\Sigma^*(n,2)$), then it should be much smaller than the extremal size. The parameter with the which we measure this difference is the number of sets in $\cF$ that do not belong to the closest extremal family, i.e. $\min|\cF\setminus \cF^*|$ where the minimum is taken over the families in $\Sigma^*(n,2)$.

\begin{lemma}
\label{lem:spernerstab}
For every $\varepsilon > 0$, there exists an $n_0$ such that the following holds: if $n\ge n_0$, $m\le \binom{(1-\varepsilon)n}{\lfloor n/2\rfloor}$ and $\cF \subseteq 2^{[n]}$ is a $2$-Sperner family with the property that for any $\cF^*\in \Sigma^*(n,2)$ we have $|\cF \setminus\cF^*| \ge m$, then the following upper bound holds on the size of $\cF$: 
$$|\cF|\le \Sigma(n,2)-g(\lceil n/2\rceil +1,m).$$
\end{lemma}

\begin{proof}
Let $\varepsilon>0$ be fixed and let $\cF \subseteq 2^{[n]}$ be a $2$-Sperner family such that for any $\cF^*\in \Sigma^*(n,2)$ we have $|\cF \setminus\cF^*| \ge m$. Note that by \pref{change} (f) we have $g(\lceil n/2\rceil +1,m) \le 2nm \le \binom{(1-0.95\varepsilon)n}{n/2}$ if $n$ is large enough. Write $m'=\min_{\cF^*\in \Sigma^*(n,2)}\{|\cF \setminus\cF^*|\}$. We can assume that $m'\le (\frac{1}{2}+o(1))\binom{n}{\lceil n/2\rceil+1}$ as otherwise $|(\binom{[n]}{\lfloor n/2\rfloor}\cup \binom{[n]}{\lceil n/2\rceil}) \setminus \cF|\ge \delta\binom{n}{\lfloor n/2\rfloor}$ would hold for some positive $\delta$ and we would be done by \cref{weakstab} part (a) or (c) depending on the parity of $n$.

\vskip 0.3truecm

\textsc{Case I.} $m' \ge \binom{(1-\varepsilon/2)n}{n/2}$.

\vskip 0.1truecm

If $n$ is odd, then by \cref{weakstab} (b), we have $|\cF|\le \Sigma(n,2)-\frac{1.9m'}{n}$ and we are done as $\frac{1.9m'}{n} \ge \frac{1.9}{n}\binom{(1-\varepsilon/2)n}{n/2} \ge \binom{(1-3\varepsilon/4)n}{n/2}$ for $n$ large enough.

If $n$ is even, then by symmetry we can suppose that $m'=|\cF\setminus (\binom{[n]}{n/2-1}\cup\binom{[n]}{n/2})|$.  Let $\cF=\cF_1 \cup \cF_2$ with $\cF_1=\{F \in \cF: \not\exists F'\in \cF, F' \subset F\}$ and $\cF_2=\cF\setminus \cF_1$. Let us write 
\[
\cF_{1,+}=\{F \in\cF_1: |F|>n/2\}, \cF_{2,+}=\{F \in\cF_2: |F|>n/2\}, \cF_-=\{F \in \cF: |F| <n/2-1\},
\]
\[
\cG_{n/2}=\{G \notin \cF, |G|=n/2\}, \hskip 0.3truecm \cG_{n/2-1}=\{G \notin \cF, |G|=n/2-1\}.
\]
Observe the following bounds:
\begin{itemize}
\item
$|\cG_{n/2}|\le \binom{(1-3\varepsilon/4)n}{n/2}$ as otherwise by \cref{weakstab} (d), we are done.
\item
$|\cF_{1,+}|\le \binom{(1-3\varepsilon/4)n}{n/2}$ as $\Delta_{n/2}(\cF_{1,+})\subseteq \cG_{n/2}$ and $|\cF_{1,+}|\le |\Delta_{n/2}(\cF_{1,+})|$ hold.
\item
$|\cF_-|\le \binom{(1-3\varepsilon/4)n}{n/2}$ as otherwise by \cref{weakstab} (c), we are done.
\item
By definition all sets in $\Delta_{n/2}(\cF_{2,+})\setminus \cG_{n/2}$ must belong to $\cF_1$. No set below an arbitrary set of $\cF_1$ belongs to $\cF$, therefore all sets of $\Delta_{n/2-1}(\cF_{2,+})$ belong to $\cG_{n/2-1}$ except those whose complete shade belongs to $\cG_{n/2}$. By double counting pairs $(G,G')$ with $G' \subset G$, $|G'|=n/2-1$, $G\in \cG_{n/2}$ and $\nabla_{n/2}(G') \subseteq \cG_{n/2}$ we obtain that the number of such exceptional sets is $(1+o(1))|\cG_{n/2}|\le \binom{(1-2\varepsilon/3)n}{n/2}$. Let $\cE$ denote the family of these exceptional sets.
\end{itemize}
Let $m''=|\cF_{2,+}|$. By the above, we have $m'\ge m''=m'-|\cF_-|-|\cF_{1,+}| \ge \binom{(1-0.6\varepsilon)n}{n/2}$. Also, writing $m''=\binom{x''}{n/2+1}$ we have
\begin{align*}
|\cF| & \le  \Sigma(n,2)-|\Delta_{n/2-1}(\cF_{2,+})|+|\cE|+|\cF_{2,+}|+|\cF_{1,+}|+|\cF_-| \\
& \le  \Sigma(n,2)-|\Delta_{n/2-1}(\cF_{2,+})|+m''+3\binom{(1-2\varepsilon/3)n}{n/2} \\
& \le  \Sigma(n,2)-\left(\binom{x''}{n/2-1}-\binom{x''}{n/2+1}\right)+3\binom{(1-2\varepsilon/3)n}{n/2} \\
& \le  \Sigma(n,2)-\frac{1}{n^2}m''+3\binom{(1-2\varepsilon/3)n}{n/2}.
\end{align*}
Here the third inequality follows by \lref{genshadow} and the last one follows from $x''\le 2n-1+o(1)$ since $m''\le m' \le (\frac{1}{2}+o(1))\binom{n}{n/2}$.
We are done as $m''\ge \binom{(1-0.6\varepsilon)n}{n/2}$ holds.
\vskip 0.5truecm

\textsc{Case II.} $m' < \binom{(1-\varepsilon/2)n}{n/2}$

\vskip 0.1truecm

Again we may assume that $|(\binom{[n]}{n/2-1}\cup \binom{[n]}{n/2}) \setminus \cF|=m'$. Let $\cF=\cF_1 \cup \cF_2$ with $\cF_1=\{F \in \cF: \not\exists F'\in \cF, F' \subset F\}$ and $\cF_2=\cF\setminus \cF_1$. Let us write 
$$\cF_{1,-}=\{F \in \cF_1: |F|<n/2-1\},\quad \cF_{1,+}=\{F \in \cF_1: |F|>n/2-1\},$$ 
$$\cF_{2,-}=\{F \in \cF_1: |F|<n/2\},\quad \cF_{2,+}=\{F \in \cF_1: |F|>n/2\}.$$ 

To bound the size of $\cF_1$ note that $\cF_1$ is disjoint both from $\Delta_{n/2-1}(\cF_{1,+})$ and $\nabla_{n/2-1}(\cF_{1,-})$. Similarly, $\cF_2$ is disjoint both from $\Delta_{n/2}(\cF_{2,+})$ and $\nabla_{n/2}(\cF_{2,-})$. By \lref{genshadow} we obtain
\begin{align*}
\label{sumi}
|\cF|=|\cF_1|+|\cF_2| & \le  \binom{n}{n/2-1}-g(n/2+2,|\cF_{1,-}|)-g(n/2,|\cF_{1,+}|)+ \\
& +  \binom{n}{n/2}-g(n/2+1,|\cF_{2,-}|)-g(n/2+1,|\cF_{2,+}|)\\
& \le  \Sigma(n,2)-g(n/2+1, |\cF_{2,-}|+|\cF_{2,+}|+|\cF_{1,-}|)-g(n/2,|\cF_{1,+}|),
\end{align*}
where we used \pref{change} (c) and (d). Let us partition $\cF_{1,+}$ into $\cF_{1,+,n/2} \cup \cF_{1,+,+}$ with $\cF_{1,+,n/2}=\{F \in \cF_{1,+}: |F|=n/2\}$. As $\cF_{1,+}$ is Sperner, $\cF_{1,+,n/2}$ and $\Delta_{n/2}(\cF_{1,+,+})$ are disjoint and thus $|\cF_{1,+,n/2} \cup \Delta_{n/2}(\cF_{1,+,+})| \ge |\cF_{1,+}|+g(n/2+1,|\cF_{1,+,+}|)$. Also $s=|\cF_{1,+,n/2} \cup \Delta_{n/2}(\cF_{1,+,+})|\le \binom{n-1}{n/2}$ and thus $g(s)\ge 0$ holds. Therefore, we obtain $|\cF_1| \le \binom{n}{n/2-1}-g(n/2+1,|\cF_{1,+,+}|)-g(n/2+2,|\cF_{1,-}|)$. By \pref{change} (c), this strengthens the above arrayed inequality to
\[
|\cF| \le \Sigma(n,2)-g(n/2+1,|\cF_{1,+,+}| + |\cF_{1,-}\cup \cF_{2,-}|+ |\cF_{2,+}|).
\]
Note that $m\le |\cF_{1,+,+} \cup \cF_{1,-}\cup \cF_{2,-} \cup \cF_{2,+}|$ as $\cF_{1,+,n/2} \subseteq \binom{n}{n/2}$. Therefore, we are done by \pref{change} (e).
\end{proof}

Having proved \lref{spernerstab} we can now turn our attention to butterfly-free families containing chains of length 3. Our main tool to bound their size is the following LYM-type inequality.

\begin{lemma}
\label{lem:imprlym}
Let $\cF \subset 2^{[n]} \setminus \{\emptyset, [n]\}$ be a butterfly-free family and let $\cM$ be defined as $\{M \in \cF: \exists F,F' \in \cF \hskip 0.3truecm F \subset M \subset F'\}$.
\[
\sum_{F \in \cF }\frac{1}{\binom{n}{|F|}}+\sum_{M \in \cM}\left(1-\frac{n}{|M|(n-|M|)}\right)\frac{1}{\binom{n}{|M|}} \leq 2.
\]
\end{lemma}

\begin{proof}
We count the pairs $(F,\cC)$ where $\cC$ is a maximal chain in $[n]$ and $F \in \cF\cap \cC$ holds. For fixed $F$ there are $|F|!(n-|F|)!$ maximal chains containing $F$. For any maximal chain $\cC$, we have $|\cF \cap \cC|\le 3$ as a 4-chain is a butterfly. If $|\cC\cap \cF|=3$, then $\cC$ contains exactly one member $M \in \cM$ as otherwise $\cF$ would contain a 4-chain. Note that for any $M \in \cM$ there exist unique sets $F_{1,M},F_{2,M}\in \cF$ with $F_{1,M} \subset M \subset F_{2,M}$. Indeed, sets with these containment properties exist by definition of $\cM$ and $M$ cannot contain two sets $F',F''\in \cF$ as $F_{2,M},M, F',F''$ would consititute a butterfly. Similarly if $M \subset F^*,F^{**}\in \cF$ holds, then $F_{1,M},M,F^*,F^{**}$ would consititute a butterfly.
Therefore, all maximal chains $\cC$ that contain $M$ with $|\cF \cap \cC|=3$ must contain $F_{1,M}$ and $F_{2,M}$ and thus their number is at most $(|M|-1)!(n-|M|-1)!$. (Here we used that $\emptyset, [n] \notin \cF$.) Moreover, for any maximal chain $\cC$ with $M \in \cC, F_{1,M},F_{2,M} \notin \cC$, we have $|\cC \cap \cF|=1$ and the number of such chains is at least $(|M|!-(|M|-1)!)((n-|M|)!-(n-|M|-1)!)$. We obtained the following inequality
\begin{align*}
\sum_{F \in \cF} |F|!(n-|F|)! & \le 2n!+\sum_{M\in \cM}(|M|-1)!(n-|M|-1)! - \\
& - \sum_{M \in \cM}((|M|-1)(|M|-1)!)((n-|M|-1)(n-|M|-1)!).
\end{align*}
Rearranging and dividing by $n!$ we obtain the claim of the lemma. 
\end{proof}

\begin{corollary}
\label{cor:butt}  Let $\cF \subseteq 2^{[n]}$ be a butterfly-free family with $\emptyset, [n] \notin\cF$ and let us write $\cM=\{M \in \cF: \exists F,F' \in \cF \hskip 0.3truecm F \subset M \subset F'\}$. If $n$ is large enough, then $|\cF| \le \Sigma(n,2)-9|\cM|/20$.
\end{corollary}

\begin{proof}
As $\emptyset, [n] \notin\cF$, for any $M \in \cM$ we have $2 \le |M|\le n-2$ and thus $1-\frac{n}{|M|(n-|M|)}\ge 9/20$ for every $M \in\cM$ if $n$ is large enough. Therefore, the two summands in \lref{imprlym} corresponding to a set $M \in \cM$ is at least $29/20$ as much as the summand corresponding to a set $F \in \cF\setminus \cM$ with $|F|=|M|$. The number of possible summands in \lref{imprlym} is $\Sigma(n,2)$ if $\cM=\emptyset$ and each pair of sets in $\cM$ leaves place for one less summand.
\end{proof}

Now we are ready to prove \tref{butterflystab} the statement of which we recall here below.

\begin{theorem:repeat}
Let $m$ be a non-negative integer with $m\le \binom{\frac{2n}{3}-1}{\lceil n/2\rceil}$ and let $\cF \subseteq 2^{[n]}$ be a butterfly-free family such that $|\cF \setminus \cF^*|\ge m$ for every $\cF^*\in \Sigma^*(n,2)$. Then the inequality $|\cF| \le \Sigma(n,2)-\frac{m}{4}$ holds if $n$ is large enough.
\end{theorem:repeat}
\begin{proof}
Let $\cF \subseteq 2^{[n]}$ be a butterfly-free family satisfying the conditions of the theorem. If $\emptyset \in \cF$ or $[n] \in \cF$, then $\cF\setminus \{\emptyset\}$ or $\cF\setminus \{[n]\}$ does not contain the poset $\vee$ or $\wedge$, where $\vee$ is the poset with three elements one smaller than the other two and $\wedge$ is the poset with three elements one larger than the other two. In either case by a theorem of Katona and Tarj\'an \cite{KT}, we have $|\cF| \le (1+O(\frac{1}{n}))\binom{n}{n/2}$. Thus we may assume $\emptyset, [n] \notin \cF$. If $\cM=\{M \in \cF: \exists F,F' \in \cF \hskip 0.3truecm F \subset M \subset F'\}$ contains at least $10m/19$ sets, then we are done by \cref{butt}.  If $|\cM| \le 10m/19$, then $\cF\setminus \cM$ is 2-Sperner and $|(\cF\setminus \cM) \setminus \cF^*|\ge 9m/19$ for every $\cF^*\in \Sigma^*(n,2)$ and thus by \lref{spernerstab} we obtain $|\cF\setminus \cM| \le \Sigma(n,2) -g(\lceil n/2\rceil +1,9m/19)\le  \Sigma(n,2)-18m/19$ as we can use \pref{change} (f) by the assumption on $m$. Therefore, $|\cF|\le \Sigma(n,2)-8m/19$ holds.
\end{proof}

\section{Proof of the main result}

In this section we prove \tref{main}. Our main tool is the stability result \tref{butterflystab} proved in the previous section. This tells us that if the structure of a butterfly-free family is very different from that of the extremal family, then it contains much fewer sets. Since any new set yields an additional copy of the butterfly poset, a family with few butterflies must contain an almost extremal butterfly-free family. To deal with families $\cF$ containing almost extremal butterfly-free families $\cG$, we have to prove that most sets in $\cF \setminus \cG$ behave very similarly to the extra sets in the conjectured extremal families. We formalize this handwaving statement in the following theorem.

\begin{theorem}
\label{thm:wrongsets} For any $\varepsilon >0$ there exists an $n_0$ such that for any $n \ge n_0$ the following holds provided $m$ satisfies $\log m=o(n)$ and $n/2- \sqrt{n}\le k \le n/2+ \sqrt{n}$: let $\cF \subset \binom{[n]}{k}$ with $|\cF|=\binom{n}{k}-m$. Then the number of sets in $\binom{[n]}{\ge k+1}$ that contain fewer than $(1-\varepsilon)k$ sets from $\cF$ is $o(m)$.
\end{theorem}

Before we start the proof of \tref{wrongsets} let us introduce some notation and an \textit{isoperimetric problem} due to Kleitman and West (according to Harper \cite{H}). Given a graph $G$ and a positive integer $m\le |V(G)/2|$, the isoperimetric problem asks for the minimum number of edges $e(X,V(G)\setminus X)$ that go between an $m$-element subset $X$ of $V(G)$ and its complement. For regular graphs, this problem is equivalent to finding the the maximum number of edges $e(X)$ in an subgraph of $G$ induced by an $m$-subset $X$ of the vertices. Indeed, in a $d$-reguar graph we have $d|X|=2e(X)+e(X, V(G)\setminus X)$.

Kleitman and West asked \cite{Kl2,Kl3} for the solution of the isoperimetric problem in the Hamming graph $H(n,k)$ whose vertex set is $\binom{[n]}{k}$ and two $k$-subsets are connected if their intersection has size $k-1$. Harper \cite{H} introduced and solved a continuous version of this problem. (For more on isoperimetric problems see e.g. \cite{H2}.) Here we summarize some of his findings. The shift operation $\tau_{i,j}$ is a widely used tool in extremal set theory. It is defined by
\begin{align*}
\tau_{i,j}(F)=\left\{
\begin{array}{cc} 
F\setminus \{j\} \cup \{i\} & \textnormal{if}\ ~ j \in F,
i \notin F ~\textnormal{and} ~ F\setminus \{j\} \cup \{i\} \notin \cF\\
F & \textnormal{otherwise.}
\end{array}
\right.
\end{align*}
And the shift of a family is defined as $\tau_{i,j}(\cF)=\{\tau_{i,j}(F):F \in \cF \}$.

Harper proved that in the Hamming graph we have $e(\cF) \le e(\tau_{i,j}(\cF))$ for any family $\cF \subseteq \binom{[n]}{k}$ and $i,j \in [n]$. Therefore, it is enough to consider the isoperimetric problem for \textit{left shifted} families, i.e. families for which $\cF=\tau_{i,j}(\cF)$ holds for all pairs $i<j$.

The characteristic vector of a subset $F$ of $[n]$ is a $0-1$ vector $x_F$ of length $n$ with $x_F(i)=1$ if $i\in F$ and $x_F(i)=0$ if $i \notin F$. 0-1 vectors with exactly $k$ one entries are clearly in one-to-one correspondence with $\binom{[n]}{k}$. But also, one can consider non-negative integer vectors of length $k$ for any set $F \in \binom{[n]}{k}$ such that $y_F(j)=i_j-j$ where $i_j$ is the index of the $j$th one entry of $x_F$. For any set $F \in\binom{[n]}{k}$ the entries of $y_F$ are non-decreasing, as $i_j-j$ is the number of zero coordinates of $x_F$ before the $j$th 1-coordinate. Also, $0 \le y_F(1) \le y_F(2) \le \dots \le y_F(k) \le n-k$ hold. 

Such vectors form the poset $L_{k,n-k}$ under coordinatewise ordering, i.e. $L_{a,b} =\{x \in [0,b]^a: x(1)\le x(2) \le \dots \le x(a)\}$ and $x\le_{L_{a,b}} y$ if and only if $x(i) \le y(i)$ for all $1 \le i \le a$. It was shown by Harper that a family $\cF\subset \binom{[n]}{k}$ is left shifted if and only if the set $\{y_F: F \in \cF\}$ is a \textit{downset} in $L_{k,n-k}$ (a set $D$ is a downset in a poset $P$ if $d'\le_P d \in D$ implies $d' \in D$). If $F,F'$ are endpoints of an edge in $H(n,k)$, then for some $i\neq j$ we have $x_F(i)=0, x_{F'}(i)=1,x_F(j)=1,x_{F'}(j)=0$ and $x_F(l)=x_{F'}(l)$ for all $l\in [n]$, $l\neq i,j$. If $i<j$, then this means that $F'$ could be obtained from $F$ by using $\tau_{i,j}$ and therefore $y_{F'}\le_{L_{k,n-k}} y_F$ holds. Moreover, the number of edges for which $F$ is the ``upper endpoint" is $r(y_F)=\sum_{i=1}^ky_F(i)$. If $\cF$ is left shifted and $F \in \cF$, then all lower endpoints of such edges belong to $\cF$, thus the number of edges spanned by $\cF$ in $H(n,k)$ is $\sum_{F \in \cF}r(y_F)$. Therefore, the isoperimetric problem in $H(n,k)$ is equivalent to maximizing $\sum_{y \in Y}r(y)$ over all downsets $Y \subset L_{k,n-k}$ of a fixed size.

We will use only the following simple observation to prove \tref{wrongsets}. 
\begin{prop}
\label{prop:isoperi} Let $\cF \subseteq \binom{[n]}{k}$ be a family of size $m < \binom{\delta n}{\delta n/2}$, with $0\le \delta <1$ and $k>\delta n/2$. Then in $H(n,k)$ we have $e(\cF) \le \delta mn^2$.
\end{prop}

\begin{proof}
Suppose not and let $\cF$ be a left shifted counterexample and thus we have $\sum_{F \in \cF}r(y_F)\ge \delta mn^2$. Therefore, there must be an $F \in \cF$ with $r(y_F) \ge \delta n^2$. Note that for such a vector, we have $y_F(k-\delta n/2)\ge \delta n/2$ as otherwise $r(y_F) \le r(y^*)\le \delta n^2$ would hold where $y^*(i)=\delta n/2$ if $i\le k-\delta n/2$ and $y^*(i)=n-k$ if $i> k-\delta n/2$.
As $\cF$ is left shifted, the set $Y_{\cF}=\{y_F: F \in \cF\}$ is a downset in $L_{k,n-k}$. Any vector $y\in L_{k,n-k}$ with $y_i=0$ for $i\le k-\delta n/2$ and $y(i)\le \delta n/2$ for $i> k-\delta n/2$ satisfies $y \le_{L_{k,n-k}} y_F$. Therefore, all those vectors belong to $Y_{\cF}$. The number of such vectors is $\binom{\delta n}{\delta n/2}$. This contradicts the assumption $m < \binom{\delta n}{\delta n/2}$.
\end{proof}

Now we are ready to prove \tref{wrongsets}.

\begin{proof}[Proof of \tref{wrongsets}] Let $\overline{\cF}=\binom{[n]}{k} \setminus \cF$ and thus $|\overline{\cF}|=m$. We want to bound the number of sets of which the shadow is contained in $\overline{\cF}$ with the exception of at most $(1-\varepsilon)k$ sets. Let $\cG \subset 2^{[n]}$ denote the family of such sets and write $\cG_l=\{G \in \cG: |G|=l\}$. To bound $|\cG_{k+1}|$ we double count the pairs $F_1,F_2$ of sets in $\overline{\cF}$ with $|F_1 \cap F_2|=k-1$. As $\overline{\cF}$ has size $m$, by applying \pref{isoperi} with a sequence $\delta_n \rightarrow 0$, we obtain the number of such pairs is $o(mn^2)$. On the other hand for every such pair there exists at most one $G \in \cG_{k+1}$ with $F_1,F_2 \subset G$ (namely, $F_1 \cup F_2$). Thus the number of such pairs is at least $|\cG_{k+1}|\binom{\varepsilon k}{2}$. Therefore, we obtain $|\cG_{k+1}|\binom{\varepsilon k}{2}=o(mn^2)$. Rearranging and the assumption on $k$ yields that $|\cG_{k+1}|=o(m)$.

To bound $|\cG_l|$ for values of $l$ larger than $k+1$, observe that $\Delta_{k+1}(\cG_l) \subseteq \cG_{k+1}$ holds for all $l >k+1$. Let $x$ denote the real number for which $|\cG_{k+1}|=\binom{x}{k+1}$ holds. By \tref{shadow}, we obtain that $|\cG_l| \le \binom{x}{l}$ holds. By the assumption on $m$ and $k$, we see that $x=k+1+o(k)$ and thus by \pref{change} (f) we have $\binom{x}{l+1} \le \frac{1}{2}\binom{x}{l}$. This gives
\[
|\cG|=\sum_{l=k+1}^n|\cG_l|\le \sum_{l=k+1}^n\binom{x}{l} \le 2|\cG_{k+1}|=o(m). \hfill \qedhere
\]
\end{proof}

We will apply \tref{wrongsets} only with $k=\lfloor n/2\rfloor -1, \lfloor n/2\rfloor, \lfloor n/2\rfloor+1$ and $\lfloor n/2\rfloor+2$.

Now we are ready to prove our main result that we recall here below. Note that by \pref{constr} we only have to deal with families $\cF \subseteq 2^{[n]}$ that do not contain any $\cS \in \Sigma^*(n,2)$. When bounding the number of butterflies in $\cF$ we will distinguish two cases depending on $m=\min_{\cS\in \Sigma^*(n,2)}|\cS\setminus \cF'|$. In the harder case, when $m$ is small, we will only count copies $F,F_1,F_2,F_3\in \cF$ where $F\in  \binom{[n]}{\ge \lceil n/2 \rceil +1}$, $F_1 \in \binom{[n]}{\lceil n/2\rceil}$, $F_2,F_3 \in \binom{[n]}{\lceil n/2 \rceil -1}$ with $F_2,F_3\subset F_1 \subset F$ or $F\in  \binom{[n]}{\le \lceil n/2 \rceil -2}$, $F_1 \in \binom{[n]}{\lceil n/2\rceil-1}$, $F_2,F_3 \in \binom{[n]}{ \lceil n/2 \rceil}$ with $F_2,F_3\supset F_1 \supset F$.

\begin{theorem:repeatmain}
Let $\cF\subseteq 2^{[n]}$ be a family of sets with $|\cF|=\Sigma(n,2)+E$.

\textbf{(a)} If $E=E(n)$ satisfies $\log E=o(n)$, then the number of butterflies contained by $\cF$ is at least $(1-o(1))E\cdot f(n)$. 

\textbf{(b)} Furthermore, if $E \le \frac{n}{100}$, then the number of butterflies contained by $\cF$ is at least $E\cdot f(n)$.
\end{theorem:repeatmain}

\begin{proof}[Proof of \tref{main} part (a)]
Let $\cF \subseteq 2^{[n]}$ be a family containing $\Sigma(n,2)+E$ sets and let $\cF'$ be a maximum size butterfly-free subfamily of $\cF$. Let $m$ be defined by $\min_{\cS\in \Sigma^*(n,2)}|\cS\setminus \cF'|$. If $m \ge 6f(n)E$ holds, then by \tref{butterflystab} we have $|\cF'| \le \Sigma(n,2)-Ef(n)$ and thus $|\cF\setminus \cF'| \ge E(f(n)+1)$. As $\cF'$ is a maximum butterfly-free subfamily of $\cF$, every set $F \in \cF\setminus \cF'$ forms a butterfly with 3 other sets from $\cF'$. Thus the number of butterflies in $\cF$ is at least $|\cF\setminus \cF'|$. This finishes the proof if $m \ge 6f(n)E$ holds.

Suppose next that $m \le 6f(n)E$ holds. Note that as $f(n) \le n^3$, $m\le 6n^3E=o(\binom{\varepsilon n}{\varepsilon n/2})$ for any positive $\varepsilon$. Without loss of generality we can assume that $|(\binom{[n]}{\lceil n/2\rceil-1} \cup \binom{[n]}{\lceil n/2 \rceil})\setminus \cF|=m$ and thus $|\cF \setminus (\binom{[n]}{\lceil n/2\rceil-1} \cup \binom{[n]}{\lceil n/2 \rceil})|=m+E$ hold. Let us write $k=\lceil n/2 \rceil -1$ and fix an $\varepsilon >0$ and pick $\varepsilon'>0$ with the property that $(1-\varepsilon')^4/2\ge 1-\varepsilon$. Applying \tref{wrongsets} to $\cF \cap \binom{[n]}{k}$ we obtain that the family $\cF_{b, k+1}=\{F \in \cF \cap \binom{[n]}{k+1}:|\Delta_k(F) \cap \cF|\le (1-\varepsilon') k\}$ has size $o(m)$. Let us apply \tref{wrongsets} again, this time to $\cF_{g,k+1}=(\cF \cap \binom{[n]}{k+1})\setminus \cF_{b,k+1}$. We obtain that the family $\cF_{b,\ge k+2}=\{F \in \binom{[n]}{\ge k+2}: |\Delta_{k+1}(F)\cap \cF_{g,k+1}|\le (1-\varepsilon') k\}$ has size $o(m)$. With an identical argument applied to $\overline{\cF}=\{[n]\setminus F: F \in \cF\}$, one can show that the families $\cF_{b,k}=\{F \in \cF \cap \binom{[n]}{k}:|\nabla_{k+1}(F) \setminus \cF|\le (1-\varepsilon') k\}$ and
$\cF_{b,\le k-1}=\{F \in \cF\cap \binom{[n]}{\le k-1}: |\nabla_{k}(F)\cap (\cF\setminus \cF_{b,k})| \le (1-\varepsilon')k\}$ both have size $o(m)$.

Let us pick a set $F \in \cF_g=\cF\setminus (\binom{[n]}{k}\cup \binom{[n]}{k+1}\cup \cF_{b,\le k-1} \cup \cF_{b,\ge k+2})$ and note that the number of such sets is $m+E-o(m)\ge E$. 
\begin{claim}
For every $F \in \cF_g$ there exist at least $(1-\varepsilon)f(n)$ copies of the butterfly poset that contain only $F$ from $\cF\setminus (\binom{[n]}{k}\cup \binom{[n]}{k+1})$.
\end{claim}

\begin{proof}[Proof of Claim]
Assume $|F|\ge k+2$. Then as $F\in\cF_g$ there are at least $(1-\varepsilon)k$ sets $F'$ in $\Delta_{k+1}(F)\cap \cF_{g,k+1}$. For every $F' \in \cF_{g,k+1}$ we have $|\Delta_k(F') \cap \cF|\ge (1-\varepsilon')k$. Since every four-tuple $F,F',F_1,F_2$ forms a butterfly where $F_1,F_2 \in \Delta_k(F') \cap \cF$ we obtain that the number of butterflies containing only $F$ from $\cF\setminus (\binom{[n]}{k}\cup \binom{[n]}{k+1})$ is at least $(1-\varepsilon')k\binom{(1-\varepsilon')k}{2}\ge (1-\varepsilon')^4k^3/2\ge (1-\varepsilon)f(n)$ if $n$ and thus $k$ are large enough. The proof of the case when $|F|\le k-1$ is similar.
\end{proof}
The above claim finishes the proof of \tref{main} part (a).
\end{proof}
\vskip 0.5truecm

To obtain part (b) of \tref{main} we need better bounds on the number of ``bad sets''. We start with the following folklore proposition.

\begin{prop}
\label{prop:folklore}
Let $U_1,\dots, U_l$ be sets of size $u$ such that $|U_i \cap U_j|\le 1$ holds for any $1\le i<j\le l$. Then we have $|\bigcup_{i=1}^lU_i|\ge l\cdot \frac{2u-l}{2}$.
\end{prop}

\begin{proof}
By the condition on the intersection sizes we have $|U_i\setminus \bigcup_{j=1}^{i-1}U_j|\ge u-i+1$ and thus  $|\bigcup_{i=1}^lU_i| \ge \sum_{i=1}^lu-i+1$.
\end{proof}

\begin{corollary}
\label{cor:wrongsetsplus} Let $\cF \subset \binom{[n]}{k}$ with $|\cF|=\binom{n}{k}-m$. Then the number of sets $G$ in $\binom{[n]}{k+1}$ that contain fewer than $k+1-2\sqrt{m}$ sets from $\cF$ is at most $\sqrt{m}$ provided $m \le k^2$ and $n/2- \sqrt{n}\le k \le n/2+ \sqrt{n}$. The number of such sets from $\binom{[n]}{\ge k+1}$ is at most $2\sqrt{m}$.
\end{corollary}

\begin{proof}
For any set $G\in\binom{[n]}{k+1}$ with $|\Delta_k(G)\cap \cF|<k+1-2\sqrt{m}$ one can consider a family $\cH_G$ of $2\sqrt{m}$ sets from $\Delta_k(G)\setminus \cF$. Clearly, if $G'\in\binom{[n]}{k+1}$ is another set with $|\Delta_k(G)\cap \cF|<k+1-2\sqrt{m}$, then $|\cH_G \cap \cH_{G'}|\le |\Delta_k(G)\cap \Delta_k(G')| \le 1$. The sets $\cH_G$ satisfy the condition of \pref{folklore}. Thus if the number of such $G$'s is more than $\sqrt{m}$, then $|\overline{\cF}|>\sqrt{m}\cdot\frac{4\sqrt{m}-\sqrt{m}}{2}=m$ which is a contradiction.

The proof of the second statement that deals with sets of larger size is as in the proof of \tref{wrongsets}.
\end{proof}

\begin{proof}[Proof of \tref{main} part (b)]
Let $\cF \subseteq 2^{[n]}$ be a family of sets with $|\cF|=\Sigma(n,2)+E$ where $E=E_n\le \frac{n}{100}$. Let $m$ be defined by $\min_{\cS\in \Sigma^*(n,2)}|\cS\setminus \cF|$. We will write $k+1=\lceil n/2 \rceil$ and assume that $m=(\binom{[n]}{k}\cup\binom{[n]}{k+1})\setminus \cF$.
We will consider four cases with respect to $m$.

\vskip 0.3truecm

\textsc{Case I.} \hskip 0.3truecm $m \ge 6f(n)E$

Just as in the proof of \tref{main} part a), we consider a maximal butterfly-free subfamily $\cF' \subseteq \cF$ with $(\binom{[n]}{k}\cup\binom{[n]}{k+1})\cap \cF \subseteq \cF'$. By \cref{butt}, $|\cF'|< \Sigma(n,2)-f(n)E$ and thus $\cF$ contains at least $|\cF|-|\cF'|>Ef(n)$ copies of the butterfly poset.

\vskip 0.3truecm

\textsc{Case II.} \hskip 0.3truecm $\frac{n}{10} \le m <6f(n)E$

We again repeat the argument of part (a). By applying \tref{wrongsets} twice with $\varepsilon=1/4$, we obtain that for $E+m-o(m)\ge (11-o(1))E$ sets $F \in \cF \setminus (\binom{[n]}{k}\cup\binom{[n]}{k+1})$ the number of copies of the butterfly poset that contains only $F$ from $\cF \setminus (\binom{[n]}{k}\cup\binom{[n]}{k+1})$ is at least $(\frac{27}{64}-o(1))f(n)$ and thus the number of butterflies in $\cF$ is much larger than $f(n)E$.

\vskip 0.3truecm

\textsc{Case III.} \hskip 0.3truecm $50 \le m <\frac{n}{10}$

We try to imitate the proof of the second case of part (a). Applying \cref{wrongsetsplus} to $\cF \cap \binom{[n]}{k}$ we obtain that the family $\cF_{b, k+1}=\{F \in \cF \cap \binom{[n]}{k+1}:|\Delta_k(F) \setminus \cF|\le  k+1 -2\sqrt{m}\}$ has size at most $\sqrt{m}$. Let us apply \cref{wrongsetsplus} again, this time to $\cF_{g,k+1}=(\cF \cap \binom{[n]}{k+1})\setminus \cF_{b,k+1}$. We obtain that the family $\cF_{b,\ge k+2}=\{F \in \binom{[n]}{\ge k+2}: |\Delta_{k+1}(F)\cap \cF_{g,k+1}|\le k+2-2\sqrt{m}\}$ has size ar most $2\sqrt{m+\sqrt{m}}\le 3\sqrt{m}$. With an identical argument applied to $\overline{\cF}=\{[n]\setminus F: F \in \cF\}$, one can show that the families $\cF_{b,k}=\{F \in \cF \cap \binom{[n]}{k}:|\nabla_{k+1}(F) \setminus \cF|\le n-k-2\sqrt{m}\}$ and
$\cF_{b,\le k-1}=\{F \in \cF\cap \binom{[n]}{\le k-1}: |\nabla_{k}(F)\cap (\cF\setminus \cF_{b,k})| \le n-k+1-2\sqrt{m}\}$ both have size at most $3\sqrt{m}$.

Let us pick a set $F \in \cF_g=\cF\setminus (\binom{[n]}{k}\cup \binom{[n]}{k+1}\cup \cF_{b,\le k-1} \cup \cF_{b,\ge k+2})$ and note that the number of such sets is at least $m+E-6\sqrt{m}$. The number of copies of butterflies in $\cF$ with $F$ being the only member from $\cF \setminus (\binom{[n]}{k}\cup \binom{[n]}{k+1})$ is at least $(k+2-2\sqrt{m})\binom{k+1-2\sqrt{m}}{2}$, and thus the number of butterflies in $\cF$ is at least
\[
(E+m-6\sqrt{m})(k+2-2\sqrt{m})\binom{k+1-2\sqrt{m}}{2} \ge   (E+\sqrt{m})(k+2-2\sqrt{m})\binom{k+1-2\sqrt{m}}{2}
\]
\[
\ge   Ef(n)+\sqrt{m}\cdot \frac{k^3}{4}-E\frac{\sqrt{m}(k+2)^2}{2} >  Ef(n),
\]
where we used $m-6\sqrt{m}\ge\sqrt{m}$ as $m\ge 50$, $k-2\sqrt{m}=(1-o(1))k$ as $m\le \frac{n}{10}$ and also $E\le \frac{n}{100}$.

\vskip 0.3truecm

\textsc{Case IV.} \hskip 0.3truecm $0< m<50$

In this case, every set in $\binom{[n]}{\ge k+2}$ contains at least $k+2-m$ sets from $\cF \cap \binom{[n]}{k+1}$ and every set in $\binom{[n]}{\ge k+1}$ contains at least $k+1-m$ sets from $\cF \cap \binom{[n]}{k}$. Similar statements hold for sets in $\binom{[n]}{\le k}$ and $\binom{[n]}{k-1}$. Therefore, all $E+m$ sets $F$ of $\cF \setminus (\binom{[n]}{k}\cup\binom{[n]}{k+1})$ are contained in at least $(k+2-m)\binom{k+1-m}{2}$ butterflies that contain only $F$ from $\cF \setminus (\binom{[n]}{k}\cup\binom{[n]}{k+1})$. Thus the number of butterflies in $\cF$ is at least
$(E+m)(k+2-m)\binom{k+1-m}{2}\ge E(k+2)(k+1)k/2+m(k+2)(k+1)k/2-7Em(k+2)^2/4$. This is strictly larger than $Ef(n)=E(k+2)(k+1)k$ as $E\le n/10\le k/4$.

\vskip 0.3truecm

The case when $m$ equals 0, was dealt with in the introduction by \pref{constr}.
\end{proof}

\section{Concluding remarks}

As mentioned in the introduction, if we add sets $G_1,G_2,\dots ,G_E \in \binom{[n]}{\lceil n/2\rceil+1}$ to the family $\cF=\binom{[n]}{\lceil n/2\rceil-1}\cup \binom{[n]}{\lceil n/2\rceil}$ with the property that $|G_i \cap G_j| \le \lceil n/2 \rceil-1$ holds for all pairs $1\le i<j\le E$, then the number of butterflies in $\cF \cup \{G_1,G_2, \dots, G_E\}$ is $Ef(n)$. The size of a largest family of sets with this property is denoted by $K(n, \lceil n/2\rceil +1)$. We propose the following conjecture.

\begin{conjecture}
\label{conj:conj} Let $E=E(n)\le K(n,\lceil n/2\rceil +1)$. If $n$ is large enough, then the minimum number of butterflies a family $\cF \subset 2^{[n]}$ of size $\Sigma(n,2)+E$ must contain is $Ef(n)$.
\end{conjecture}

If $k$ tends to infinity, determining the asymptotics of $K(n,k)=\max_{\cG \subset \binom{[n]}{k}}\{|\cG|: |G_1 \cap G_2|\le k-2 \hskip 0.2truecm\forall G_1,G_2 \in\cG\}$ is one of the most important open problems in coding theory \cite{CHLL}. When $k$ is roughly $n/2$, then by a trivial volume argument one has that $K(n,k) \ge (4-o(1))\binom{n}{k}/n^2$ and by a probabilistic argument one can obtain $K(n,k) = O(\frac{\binom{n}{k}\log n }{n^2})$.
The bounds in almost all the steps towards \tref{main} can be improved with a little work, but we do not see how to get anywhere near $K(n,\lceil n/2\rceil +1)$ even for the asymptotic statement.

\vskip 0.2truecm

An important step in the proof of \tref{butterflystab} is \cref{butt} that bounds the size of a butterfly-free family as a function of the number of missing middle sets. The bound of \cref{butt} is most probably very far from being sharp. We propose the following conjecture.

\begin{conjecture}
Let $\cF \subseteq 2^{[n]}$ be a butterfly-free family with $\emptyset, [n] \notin\cF$ and let us write $\cM=\{M \in \cF: \exists F,F' \in \cF \hskip 0.3truecm F \subset M \subset F'\}$. Then $|\cF| \le \Sigma(n,2)-\Omega(n|\cM|)$ if $|\cM|$ is polynomial in $n$.
\end{conjecture}

The exact value of $La(n,P)$ is known for very few posets $P$. There is a general conjecture \cite{GL,GrL} that $La(n,P)/\binom{n}{\lfloor n/2 \rfloor}$ tends to an integer $e(P)$, where the value of $e(P)$ is defined to be the largest integer such that a family from $\Sigma^*(n,e(P))$ is $P$-free. One of the nicest results in the area was proved by Bukh \cite{B} and states that this conjecture is true for posets $T$ of which the \textit{Hasse diagram} is a tree (the Hasse diagram of a poset $P$ is the directed graph of which the vertices are elements of $P$ and $p$ and $q$ are joined by an edge directed towards $q$ if $p$ preceeds $q$, i.e. $p<_P q$ and there does not exist  $z\in P$ with $p<_P z <_P q$). Bukh's result was extended by Boehnlein and Jiang \cite{BJ} to \textit{induced} $T$-free families, i.e. families $\cF$ for which there does not exist an injection $i:T \rightarrow \cF$ with $t_1\le_T t_2$ if and only if $i(t_1) \subseteq i(t_2)$. Forbidden induced subposet results are very rare. Boehlein and Jiang's theorem suggests that tree posets are easier to handle than general posets. As the exact value of $La(n,T)$ is not known even for tree posets $T$, it is not easy to formulate a conjecture on the minimum number of copies of $T$ in a family of size $La(n,T)+E$, but maybe the following structural statement can have a pushing-to-the-middle proof that is not very complicated.

\begin{conjecture}
Let $T$ be a poset of which the Hasse diagram is a tree. Then there exists a constant $c=c(T)$ with $0<c\le 1$ such that if $E \le c\binom{n}{\lfloor n/2\rfloor}$, then there exists a family $\cF\subseteq 2^{[n]}$ that minimizes the number of copies of $T$ over families of size $La(n,T)+E$ and contains sets only of $e(T)+1$ different sizes and the set sizes are consecutive integers.
\end{conjecture}

Let us comment on the whole area of forbidden subposet problems. Almost all proofs of the known results use double counting or a variant of Katona's permutation method. A nice consequence of this is that these proofs are very elegant and relatively easy to present even to an audience less familiar with extremal set systems. However, using only these tools the above mentioned conjecture, $La(n,P)/\binom{n}{\lfloor n/2 \rfloor}\rightarrow e(P)$, seems out of reach. It would be foolish to think that the Kleitman-West problem will turn out to be helpful in proving this conjecture. However, it is very likely that establishing relations to other extremal problems can be of great use. This is what Methuku and P\'alv\"olgyi did \cite{MP} when they proved their recent result that for every poset $P$ there is a constant $c_P$ such that every \textit{induced} $P$-free family $\cF \subseteq 2^{[n]}$ we have $|\cF| \le c_P\binom{n}{n/2}$. Finally, if one tries to find an inductive proof to the $La(n,P)/\binom{n}{\lfloor n/2 \rfloor}\rightarrow e(P)$ conjecture, then one might try to use supersaturation results: if by induction we know that a family $\cF$ contains \textit{many} copies of some $P' \subseteq P$, then we could try to glue some of these copies together to obtain a copy of $P$ in $\cF$.

\section*{Acknowledgement}
Research supported by the J\'anos Bolyai Research Scholarship of the Hungarian Academy of Sciences. I would like thank three anonymous referees for their thorough reading and their remarks that helped to improve the presentation of the paper.

\end{document}